\documentclass[a4paper,10pt]{article}
\usepackage[utf8]{inputenc}
\usepackage{color}
\usepackage[english]{babel}
\usepackage{amsmath,amsthm, amssymb}
\usepackage{amsfonts}
\usepackage{authblk}
\newtheorem{thm}{Theorem}
\newtheorem{cor}{Corollary}
\newtheorem{lem}{Lemma}[section]
\newtheorem{prop}{Proposition}
\usepackage{varioref}

\theoremstyle{definition}

\newtheorem{question}{Question}

\newtheorem{defn}[section]{Definition}
\newtheorem{rem}{Remark}[section]

\title{Integrability of \( C^1\) invariant splittings }

\author[1]{Stefano Luzzatto, Sina Tureli, Khadim War}

\date{March 26, 2015}

\begin{document}
\maketitle

\begin{abstract}
We derive some new conditions for integrability of dynamically defined $C^1$ invariant splittings in arbitrary dimension and co-dimension. In particular we prove that every 2-dimensional \(  C^{1}  \) invariant decomposition on a 3-dimensional manifold satisfying a volume domination condition is uniquely integrable. In the special case of volume preserving diffeomorphisms we show that standard dynamical domination is already sufficient to guarantee unique integrability.  
\end{abstract}

\section{Introduction and Statement of Results}
Let \( M \) be a  smooth manifold and \( E\subset TM \)
a  distribution of tangent hyperplanes.  A classical problem concerns the \emph{(unique) integrability} of such a distribution, i.e. the existence at every point of a (unique) local embedded submanifold  everywhere tangent to the given distribution
%\footnote{If a distribution is uniquely integrable than it defines a \emph{foliation} of \( M \), i.e.  a partition of \(  M  \) by immersed
%complete  (in the induced topology) \(  k  \)-dimensional submanifolds, called \emph{leaves} with the additional
%property that around each $x\in M$, there is a chart $\phi: U \rightarrow \mathbb{R}^{d}$ such that the
%image of each leaf intersects the chart in
%at most countable many $k$ dimensional disks \cite{Lee}. }.
 For one-dimensional distributions this is simply
 the existence and uniqueness of solutions of ordinary differential equations which is well known to hold  for
 Lipschitz distributions on compact manifolds. In higher dimensions,
 however, the question is highly non trivial and generally false  even if the distribution is smooth \cite{Lee}. 
 Some results can be obtained  however if the distribution \( E \) is associated to a diffeomorphism satisfying certain conditions.

 \subsection{Dynamical domination}
 Let \( \phi: M \to M \) be a \( C^2 \)
 diffeomorphism and  \(
 E \oplus F
 \)   a \( D\phi \)-invariant \( C^1 \) decomposition. 
 For \( x\in M \) we let \( \|D\phi_x|_E\|=\max_{v\in E}\|Df_x(v)\|/\|v\| \) denote the operator norm of \( D\phi_x \) restricted to \( E \), and let \( m(D\phi_x|_F) =\min_{v\in F, v\neq 0}\|D\phi_x(v)\|/\|v\|   \) denote the co-norm of \( D\phi_x|_F \).
The decomposition  \(
 E \oplus F
 \) is \emph{dynamically dominated} if 
 \begin{equation}\label{dyndom}
 \frac{\|D\phi_x|_E\|}{m(D\phi_x|_F)}<1
 \end{equation}
 for every \( x\in M \). Our first result is the following. 
 
 \begin{thm}\label{thm:vol}
Let  \( M \) be a 3-dimensional Riemannian manifold, \( \phi: M \to M \) a volume preserving \( C^2 \) diffeomorphism, and \( E\oplus F \) a \( D\phi \)-invariant \( C^1 \) dominated decomposition with \( dim (E) =2 \). 
 Then \( E \) is uniquely integrable. 
 \end{thm}
 
 We will obtain Theorem \ref{thm:vol} as a special case of more general results to be stated below, and while we postpone a full discussion of existing literature, we mention that this does not follow from any existing result.

  \subsection{Volume domination}
 We say that the decomposition  \( E \oplus F  \) is \emph{volume dominated} if 
\begin{equation}\label{voldom}
\frac{|det{D\phi_x|_E}|}{|det D\phi_x|_F|} < 1
\end{equation}
for all \( x\in M \) and is \emph{inverse volume dominated} if 
\begin{equation}\label{invvoldom}
\frac{|det{D\phi^{-1}_x|_E}|}{|det D\phi^{-1}_x|_F|} < 1
\end{equation}
for all \( x\in M \).

\begin{rem}
Notice that \eqref{voldom} and \eqref{invvoldom} are exclusive conditions since \eqref{invvoldom}
can be written as \( {|det{D\phi_x|_E}|}/{|det D\phi_x|_F|} > 1  \) for all \(  x\in M  \) which is the exact opposite of \eqref{voldom}. The two conditions as stated cannot be combined into a unique condition of the form \(  {|det{D\phi_x|_E}|}/{|det D\phi_x|_F|} \neq 1  \) since each one is formulated as a global condition on all points \(  x\in M  \) (though our result to be stated below actually allows significantly  more flexibility, see Remark \ref{weakdom} below).
\end{rem}

\begin{rem}\label{dynvoldom}
In  general dynamical domination and volume domination are independent conditions, 
but  in the 3-dimensional volume preserving setting with \(  dim (E) =2  \),  as in Theorem \ref{thm:vol}, dynamical domination implies volume domination.  Indeed,   the volume preservation property implies  \( {det{D\phi_x|_E}}\cdot {det D\phi_x|_F} =1\) and so \eqref{dyndom} implies \(  {det D\phi_x|_F} > 1  \) (arguing by contradiction, \(  {det D\phi_x|_F}\leq 1  \) would imply  \(  det D\phi_{x}|_{E}\geq 1  \) by the volume preservation, and this would imply \(  \|D\phi_{x}|_{E}\|/ det D\phi_{x}|_{F}\geq 1   \) which, using the fact that \(  F  \) is one-dimensional and therefore   \(  det D\phi_x|_F = m(D\phi_x|_F) \),    would contradict \eqref{dyndom}). Dividing the equation  \( {det{D\phi_x|_E}}\cdot {det D\phi_x|_F} =1\)  through by \(  ({det D\phi_x|_F})^{2}  \) we get \eqref{voldom}.
\end{rem}

By Remark \ref{dynvoldom}, Theorem \ref{thm:vol} is a consequence of the following more general  result which replaces the volume preserving property and dynamical domination with volume domination or inverse volume domination. 

\begin{thm}\label{thmdim3}  
Let  \( M \) be a 3-dimensional Riemannian manifold, \( \phi: M \to M \) a  \( C^2 \) diffeomorphism, and \( E\oplus F \) a \( D\phi \)-invariant \( C^1 \)  volume dominated or inverse volume dominated decomposition with \( dim (E) =2 \). 
 Then \( E \) is uniquely integrable. 
\end{thm}

\begin{rem}\label{weakdom}
We remark that our argument actually requires conditions which are  a-priori weaker than volume or inverse volume domination. Indeed, we just require that 
\begin{equation}\label{3dimdom}
\liminf_{k \rightarrow \infty}\frac{det{D\phi^k_x|_E}}{det D\phi^k_x|_F} =0
\quad \text{ and/or } \quad 
\liminf_{k \rightarrow \infty}\frac{det{D\phi^{-k}_x|_E}}{det D\phi^{-k}_x|_F} =0
\end{equation}
for every \(  x\in M  \). 
The two conditions in \eqref{3dimdom} are clearly formally implied by volume domination and inverse volume domination respectively\footnote{It may be that in practice \eqref{3dimdom} is actually equivalent to the volume domination condition but we have not been able to give a formal argument for this. This equivalence would be clear if we replaced the liminf by a limit since then, by compactness of \(  M  \) the convergence would be uniform in \(  x  \) and we would be able to find some \(  k  \) for which the ratio was \(  <1  \) for all \(  x  \). With the liminf condition it is not clear that this is possible.}.
Most importantly, however, we emphasize  that condition \eqref{3dimdom} allows for the possibility that both conditions hold simultaneously along different orbits, which may actually be a significant weakening of the volume domination and inverse volume domination conditions as stated above. Moreover we also remark that it is sufficient for either of these conditions to hold on a  dense subset  \( \mathcal A\subset M \). 
 \end{rem}
 
\subsection{Weak domination}
Both Theorems \ref{thm:vol} and \ref{thmdim3} are actually special cases of our most general and main result which gives conditions for integrability of distributions in arbitrary dimension and arbitrary co-dimension. 
From now on we let 
\( M \) be a compact Riemannian manifold of  \( dim(M)\geq 3 \), \( \varphi: M \to M \)  a \( C^2 \) diffeomorphism and 
\( E\oplus F
\) a \(  D\phi  \)-invariant decomposition  with \( dim(E)=d\geq 2 \) with and \(dim(F)=\ell \geq 1\). 
We generalize the volume domination and inverse domination conditions stated above as follows. 
 For every \( k \geq 1 \) we let
\(
s^k_1(x) \leq s^k_2(x) \leq .. s^k_d(x)
\) and \( 
 s^{-k}_1(x)\geq s^{-k}_2(x) \geq .. s^{-k}_d(x)
\)
 denote the singular values\footnote{We recall that the singular values of \(  D\phi^{\pm k}_x|_{E}  \) at \(  x\in M  \) are the square roots of the eigenvalues of the
self-adjoint map
\(
  (D\varphi^{\pm k}_{x}|_{E})^{\dag}\circ D\varphi^{\pm k}_{x}|_{E}:E(x) \to E(x)
\)
   where
\(   (D\varphi^{\pm k}_{x})^{\dag}|_{E}: E(\varphi^{\pm k}(x))\to E(x)  \) is the conjugate of  \(D\varphi^{\pm k}_{x}|_{E}\) with respect to the metric
\(  g  \), i.e. the unique map which satisfies \(  g(D\varphi^{\pm k}_{x}|_{E}u, v) = g(u,(D\varphi^{\pm k}_{x}|_{E})^{\dag}v)\) for all \(  u\in E(x), v\in E(\varphi^{\pm k}(x)). \)}
 of \( D\phi^k_x|_{E}\) and    \( D\phi^{-k}_x|_{E} \) respectively, and let 
\(
r^k_1(x) \leq r^k_2(x) \leq .. r^k_{\ell}(x)
\) 
and \( 
 r^{-k}_1(x)\geq r^{-k}_2(x) \geq .. r^{-k}_{\ell}(x)
\)
 denote the singular values of \( D\phi^k_x|_{F}\) and    \( D\phi^{-k}_x|_{F} \) respectively.

\begin{thm}\label{thm-main}
Let \( M \) be a compact Riemannian manifold of  \( dim(M)\geq 3 \), \( \varphi: M \to M \)  a \( C^2 \) diffeomorphism, and \( E\oplus F
\)  a
 \( D\phi \)-invariant
\( C^1 \) tangent bundle splitting
with \( dim(E)=d \geq 2 \) and \(dim(F)=\ell \geq 1\).
Suppose there exists a dense subset  \( \mathcal A\subset M \)  such that for every \( x\in \mathcal A \) we have
\begin{equation*}\tag{\ensuremath{\star}}\label{star}
\liminf_{k \rightarrow \infty}\frac{s^k_{d-1}(x)s^k_d(x)}{r^k_1(x)} =0
\quad\text{ and/or } \quad
\liminf_{k \rightarrow \infty}\frac{s^{-k}_{1}(x)s^{-k}_2(x)}{r^{-k}_\ell(x)} =0.
\end{equation*}
Then \( E \) is uniquely integrable.
\end{thm}

Notice that  \eqref{3dimdom} is equivalent to \(  (\star)  \) in the three-dimensional setting, and therefore implies Theorems  \ref{thmdim3} and \ref{thm:vol}. 
Indeed, 
the largest singular value gives the norm of the map and smallest one gives the co-norm, so one has \( s_d^k(x)=|D\phi^k_x|_E|, s_1^{-k}=|D\phi^{-k}_x|_E| , r_1^{k} = m(D\phi^k_x|_F), r^{-k}_{\ell} = m(D\phi^{-k}_x|_F) \).
So if \( dim (E)=2 \) we have 
 \( s_d^k(x)s_{d-1}^k(x)=\det (D\phi^k_x|_E) \) and
 \( s_2^{-k}(x)s_{1}^{-k}(x)=\det D\phi^{-k}_x|_E \), and if \( dim (F)=1 \) then we have \( r_{1}^{k}(x)=m(D\phi^k_x|_F) = det (D\phi^k_x|_F) \) and
 \( r_{\ell}^{-k}(x)=m(D\phi^{-k}_x|_F) = det (D\phi^{-k}_x|_F) \).
 
Theorem \ref{thm-main} generalizes existing results in \cite{WB, HHU, AH}. Indeed, integrabiilty is proved in \cite{WB} assuming that the decomposition \(  E\oplus F  \) is \(  C^{2}  \), dynamically dominated and 2-partially hyperbolic, i.e. satisfies \( m(D\phi_x|_{F}) > |D\phi_x|_{E}|^2\) and follows from the arguments in \cite{HHU} for \(  C^{1}  \) decompositions and in \cite{Par} for Lipschitz decompositions with the same dynamical assumptions. Notice that even in the 3-dimensional setting, the volume domination conditions assumed in Theorem \ref{thmdim3} are strictly weaker than the 2-partially hyperbolic condition. 
Integrability is also proved in 
 \cite{AH}  under the assumption
that there exist constants $a,b,c,d>0$ with \( [a^2,b^2] \cap [c,d] = \emptyset  \) such that \(
a|v|< |D\phi_x v| < b|v| \) for all \( v \in E(x)\) and
\(
c|v|< |D\phi_x v| < d|v|\) for all \( v \in F(x)\), which are also more restrictive that \(  (\star)  \). 
Indeed,   we have  either $b^2 < c$ or $a^2 > d$.  The first case implies that
$ ({b^2}/{c})^k$ goes to zero exponentially fast and so, letting  $x^\ell = \phi^\ell(x)$ we have
$ ({b^2}/{c})^k \geq  {(s^1_n(x^{k-1})...s^1_n(x))^2}/({r^1_1(x^{k-1} )...r^1_1(x)}) \geq  {(s^k_n)^2(x)}/{r^k_1(x)} \geq  {s^k_n(x)s^k_{n-1}(x)}/{r^k_1(x)} $
and  hence the first condition in \eqref{star} is satisfied. In exactly the same way if $a^2 > d$ holds then the second condition in \eqref{star} is satisfied.

%We emphasize that we do not assume here that \(  E  \) admits any further splitting of the form \(  E=E^{s}\oplus E^{c}  \) where \(  E^{s}  \) is uniformly contracting, even though our results are new and non-trivial even if it did (except for the  the special case where \(  E^{c}  \) is one-dimensional, which is essentially covered in \cite{AH}). Some integrability results do exist (in some cases with weaker regularity conditions on the splitting) under such an additional assumption as well as significant additional conditions concerning the topology of the ambient manifold 
%\cite{BBI, BHHTU, HHU2, HHU3, Wil} and relaxing it is not a mere technicality as integrability may fail in this case even under strong domination and smoothness, see \cite{Wil} which cites a construction of  \cite{Sma67}.
%

\subsection{Lyapunov regularity}

The domination condition \(  (\star)  \) can be adapted, in a quite interesting way, under the assumption that the set \(  \mathcal A  \) consists of \emph{regular} points\footnote{A point $x\in M$ is said to be regular for the map $\varphi$ when there is a   splitting of the tangent space $T_xM = E_1(x) \oplus E_2(x) \oplus ... \oplus E_s(x)$ invariant with respect to $D\varphi$ such that  certain conditions are satisfied such as the existence of a specific asymptotic exponential growth rate (Lyapunov exponent) is well defined, see \cite{BarPes} for precise definitions. 
The multiplicative ergodic Theorem of Oseledets \cite{} says that the set of regular points has \emph{full probability} with respect to \emph{any} invariant probability measure.  
In particular, the assumptions of the Theorem are satisfied if every open set has positive measure for some invariant probability measure, in particular this holds if \( \varphi \) is volume preserving or has an invariant probability measure which is equivalent to the volume. 
It is easy to see that the invariant subbundles $E$ and $F$ can be separately split on the orbit of $x$ using these subbundles, that is $E(x^k) = E_{i_1}(x^k) \oplus ... \oplus E_{i_d}(x^k)$, $F(x^k) = E_{j_1}(x^k) \oplus ... \oplus E_{j_{\ell}}(x^k)$ where the indice sets $\{i_1,...,i_{d}\}$ and $\{j_1,...,j_{\ell}\}$ do not intersect.  
}.

\begin{thm}\label{thm:volpres}
Let \( M \) be a compact Riemannian manifold of  \( dim(M)\geq 3 \), \( \varphi: M \to M \)  a \( C^2 \) diffeomorphism and \( TM=E\oplus F
\)  a
 \( D\phi \)-invariant
\( C^1 \) tangent bundle splitting
with \( dim(E)=d \geq 2 \) and \(dim(F)=\ell \geq 1\).   Suppose that there is a dense subset \(  \mathcal A \subset M \) of regular points
such that for each \(  x\in \mathcal A  \) and each set of indices \( 1\leq i, j\leq d, 1\leq m\leq \ell \) there exists a constant \( \lambda >0 \) such that 
\begin{equation*}\tag{\ensuremath{\star\star}}\label{starstar}
\frac{s^k_i(x)s^k_j(x)}{r^k_m(x)} \leq e^{-\lambda k}  \quad \forall \ k\geq 1 
\quad \text{or} \quad \quad \frac{s^{-k}_i(x)s^{-k}_j(x)}{r^{-k}_m(x)} \leq e^{-\lambda k} 
\quad \forall \ k\geq 1.
\end{equation*}
Then $E$ is uniquely integrable.
\end{thm}

We make a few remarks comparing conditions \(  (\star)  \) and \(  (\star\star)  \) . Notice first of all that neither of the two conditions  \(  (\star)  \) and \(  (\star\star)  \) is strictly stronger or weaker that the other. Indeed, notice first of all 
 the decay of the ratios is required to be exponential in condition \(  (\star\star)  \) which makes it a stronger requirement than the slow decay allowed by~\(  (\star)  \). On the other hand, a perhaps more interesting observation is that both conditions include an estimate which depends on the forward iterates of the map and an estimate which depends on the backward iterates of the map, but in  condition \(  (\star\star)  \) the choice of which of these two estimates to satisfy is  allowed to depend on the choice of indices. More precisely, notice that 
\begin{equation}\label{condition1}
 \frac{s^k_{d-1}(x)s^k_d(x)}{r^k_1(x)} \leq e^{-\lambda k} 
\qquad
\Longrightarrow
\qquad
\frac{s^k_i(x)s^k_j(x)}{r^k_m(x)}\leq e^{-\lambda k} 
\end{equation}
for \emph{every} set of indices \( 1\leq i, j\leq d, 1\leq m\leq \ell \), 
and similarly
 \begin{equation}\label{condition2}
\frac{s^{-k}_{1}(x)s^{-k}_2(x)}{r^{-l}_n(x)} \leq e^{-\lambda k} 
\qquad \Longrightarrow
\qquad
\frac{s^{-k}_i(x)s^{-k}_j(x)}{r^{-k}_m(x)} \leq e^{-\lambda k} 
\end{equation}
for \emph{every} set of indices \( 1\leq i, j\leq d, 1\leq m\leq \ell \). 
Thus, (an exponential version of) condition \( (\star) \) requires that  one of the equations \ref{condition1} and \ref{condition2}  be satisfied and therefore forces \(  (\star\star)  \) to hold either always in forward time for all choices of indices or always in backward time for all choices of indices. The crucial point of condition \(  (\star\star)  \) is to weaken this requirement and to allow either the backward time condition or the forward time condition to be satisfied depending on the choice of indices (in fact we will see in Section \ref{secvolpres} that under the assumption of  Lyapunov regularity these two choices are mutually  exclusive).
In particular this means that the sub-bundles \(  E  \) and \(  F  \) do not satisfy an overall domination condition but rather some sort of ``non-resonance'' conditions related to the further intrinsic Oseledets splittings of \(  E  \) and \(  F  \).

 Theorem \ref{thm:volpres} is actually just a restatement of Theorem 1.2 of Hammerlindl in \cite{AH}, whose main motivation was indeed to obtain integrability without any global domination condition. Hammerlindl formulated his result for volume preserving diffeomorphisms (but he actually only uses the density of regular points) and in terms of Lyapunov exponents, assuming that for all pairs of Lyapunov exponents $\mu_1, \mu_2$ of $D\phi|_{E}$ and $\lambda$ of $D\phi|_{F}$, we have \(  \mu_1 + \mu_2 \neq  \lambda\). In Section \ref{secvolpres} we show that these assumptions are equivalent to \(  (\star\star)  \) and thus essentialy reduce Theorem \ref{thm:volpres} to Theorem 1.2 of \cite{AH}.

\subsection{Strategy of the proof}

The core of the  paper is   the proof of Theorem \ref{thm-main} in section \ref{section2}. 
We will use a classical result of Frobenius \cite{AgrSac, Lee} which gives necessary and sufficient conditions for integrability of \( C^1 \) distributions in terms of Lie brackets of certain vector fields spanning the distribution\footnote{There exist also some partial generalisations of the Frobenius Theorem to distributions with less regularity, see for example, \cite{Simic, Ram}}. More precisely we recall that a distribution is \emph{involutive} at a point \( x \) if any two \( C^1 \) vector fields \( X, Y \in E\) defined in a neighbourhood of \( x \) satisfy \( [X,Y]_x\in E_{x} \) where \( [X,Y]_x \) denotes the Lie bracket of \( X \) and \( Y \) at \( x \). Frobenius' Theorem says that a \( C^1 \) distribution is locally uniquely integrable at \( x \) if and only if it is involutive at \( x \). Moreover it follows from the definition that non-involutivity is an open condition and thus it is sufficient to check involutivity on an a dense subset of \( M \), in particular the subset \( \mathcal A \),  to 
imply unique integrability of \( E \). 
We fix once and for all a point
\[
 x_0\in \mathcal A \]
 satisfying condition \( (\star) \),  and
let \( W, Z \) be two arbitrary \( C^1 \) vector fields in \( E \) defined in a neighbourhood of \( x_0 \). Letting \( \Pi: TM\to F \) denote the projection onto \( F \) along \( E \) we will prove that
\begin{equation}\label{zero}
 |\Pi[W, Z]_{x_0}|=0,
\end{equation}
i.e. the component of \( [W, Z]_{x_0} \) in the direction of \( F_{x_0} \) is zero, and so
 \( [W, Z]_{x_0}\in E_{x_0} \). This implies involutivity, and thus unique integrability,
 of \( E \). In the next section we will define fairly explicitly two families of local frames for \( E \) (recall that a local frame for $E$ at $x$ is a set of linearly independent vector fields which span $E_z$ for all $z$ in some neighbourhood of $x$) and show that the norm \(  |\Pi[W, Z]_{x_0}| \) is bounded above by the Lie brackets of the basis vectors of these local frames. We then show that the Lie brackets of these local frames are themselves bounded above by terms of the form  \( (\star) \) thus implying that they go to zero at least for some subsequence and thus proving
 \eqref{zero}.

Theorem \ref{thm:volpres} follows from the result \cite[Theorem 1.2]{AH} and  
Proposition \ref{prop-lyasing} which says that  the rates of growth of the singular values
  are exactly the Lyapunov exponents.

\subsection{Further questions}
Condition \((\star\star)\) is more restrictive than \( (\star) \) in the sense that it requires the decay to be exponential, but is less restrictive than \( (\star) \)  in the sense that there is no overall second order domination since it is sufficient that  for each choice of indices $i,j,m$ one of the two conditions in 
 \((\star\star)\)  be satisfied, and which one is satisfied may depend on the choice of indices. In this sense the domination behaviour of $E$ and $F$ are interlaced, that is there might exist some invariant subspaces of $F(x)$ that dominate (in a second order sense) some invariant subspaces of $E(x)$ and vice versa. With this perspective the comparison becomes more clear. An overall domination allows one to remove the condition of Lyapunov regularity and weaken exponential convergence to $0$ to convergence at any rate. While if one wants to remove over all second order domination then, within the available results and techniques, one needs to impose Lyapunov regularity
 and exponential convergence to $0$. This brings about several plausible questions:

\begin{question}
Can condition \(  (\star)  \) be replaced by condition \(  (\star\star)  \) (in the sense of removing overall domination) with an assumption
which is weaker than Lyapunov regularity?
\end{question}

\begin{question}
Can the exponential convergence condition  in \(  (\star\star)  \) be replaced by a slower or general convergence to $0$ under additional assumptions therefore generalizing theorem 1.2 \cite{AH}
\end{question}

\subsection{Acknowledgements}
We would like to thank Ra\'ul Ures for many useful discussions which have benefited the paper. 

\section{Involutivity}\label{section2}

In this Section we prove Theorem \ref{thm-main}.

\subsection{Orienting the brackets}
The idea which allows us to improve on existing results and work with a condition as weak as \( (\star) \),
depends in a crucial way on the choice of the sequence of local frames which we use to
bound \( |\Pi[W, Z]_{x_0}| \). We will rely on a relatively standard general construction
which is essentially the core of the proof of Frobenius's Theorem, see \cite{Lee}, and which is stated and proved in  \cite{AH} exactly in the form which we need here, we therefore omit the proof.

\begin{lem}\label{lem:frob}
Let $E \subset TM$ be a $C^1$ \( d \)-dimensional distribution defined in a neighbourhood of \( x_0 \).  Let $e_1,..,e_n$ be any choice of basis for $E_{x_0}$, and \( F_{x_0} \) a subspace complementary to \( E_{x_0} \). Then there exists a \( C^1 \) local  frame $\{E_i\}_{i=1}^d$ for $E$ around $x_0$ s.t $E_i(x_0) = e_i$ and \( [E_i, E_j]_{x_0}\in F_{x_0} \) for any \( 1\leq i, j \leq d \).
\end{lem}

 In what follows
 we evaluate all objects and quantities at \( x_0 \) so the reader should keep this in mind when we omit this index, when there is no risk of confusion.
 For each \( k\geq 1 \) we
let  $v^{(k)}_1,.., v^{(k)}_d$  be an orthonormal choice of eigenvectors  at
$E_{x_0}$ which span the
eigenspaces of $((D\phi^k)|_{E})^{\dag}(D\phi^k)|_{E}$ and which satisfy
$|(D\phi^k)|_{E}v^k_i| = s^k_i$,
where $s^k_1 \leq s^k_2 \leq ..$ are the singular values of the map $D\phi^k|_{E}$.
We then let \( \{Y_i^{(k)}\} \)  be the \( C^1 \) local frame given by Lemma \ref{lem:frob}, i.e.  such that
\[
 Y^k_i(x_0) = v^k_i \quad \text{ and } \quad
  [Y^k_i,Y^k_j]_{x_0} \in F_{x_0}.
\]
Now for each \( k\geq 1 \), let \( (i(k), j(k)) \) denote a (not necessarily unique) pair of indices that maximise the norm of the bracket,  i.e.
\[
| [Y^{k}_{i(k)}, Y^{k}_{j(k)}]_{x_0}|\geq  | [Y^{k}_\ell, Y^{k}_m]_{x_0} |
\]
for all \(1\leq \ell, m \leq d \). The proof for the first condition in \eqref{star} will use these vector fields.  The proof of the second condition is exactly the same by considering $D\phi^{-k}$ and using vector fields \( \{X_i^{(k)}\} \) which satisfy the similar conditions as above, that is \(X^k_i(x_0) = w^k_i \) and \([X^k_i,X^k_j]_{x_0} \in F_{x_0}\) where $w^k_i$ are a choice of orthonormal eigenvectors associated to the singular values $s^{-k}_1 \geq s^{-k}_2 \geq ..$

\subsection{A priori bounds}

The following Lemma gives some  upper
bounds on \( |\Pi[Z,W]_{x_0}| \) in terms of the Lie brackets
\( | [Y^k_{i(k)}, Y^k_{j(k)}]_{x_0}| \). These bounds are 
\emph{a priori} in the sense that they do not  depend
on the specific form of the local frame \( \{Y^{k}_i\} \), but simply on the fact that they are orthonormal. The statement we give here is thus  a special case of a somewhat more general setting.

\begin{lem}\label{lem-orthlie} For every \( k\geq 1 \) we have
$$|\Pi[Z,W]_{x_0}| \leq d (d-1) |[Y^{k}_{i(k)},Y^{k}_{j(k)}]_{x_0}|$$
\end{lem}

\begin{proof}

Write
\[
Z = \sum_{\ell=1}^{d}\alpha^{(k)}_\ell Y^{k}_\ell
\quad \text{ and } \quad
W = \sum_{m=1}^{d}\alpha^{(k)}_m Y^{k}_m
\]
 for some
functions
$\alpha^{(k)}_l, \alpha^{(k)}_l$ which, by orthogonality of the frames at $x_0$, satisfy
$|\alpha^{(k)}_m(x_0)|,|\alpha^{(k)}_m(x_0)|\leq1$.
We have by bilinearity of $[\cdot,\cdot]$:
 \[
 [Z, W] =
\sum_{l,m=1}^{d}\alpha^{(k)}_\ell\alpha^{(k)}_m [Y^{k}_\ell,Y^{k}_m] + \alpha^{(k)}_\ell  Y^{k}_\ell
(\alpha^{(k)}_m) Y^{k}_m - \alpha^{(k)}_m Y^{k}_m (\alpha^{(k)}_\ell )Y^{k}_\ell
 \]
Applying the projection \( \Pi \) to both sides and using the fact that
 $\Pi(Y^{k}_\ell)=\Pi(Y^{k}_m)=0$, $[Y^{k}_\ell,Y^{k}_m]_{x_0} \in F(x_0)$, and taking norms,  we have
\[
|\Pi[Z,W]_{x_0}| \leq
\sum_{\ell,m=1}^{d}|\alpha^{(k)}_\ell\alpha^{(k)}_m(x_0)  |[Y^{k}_\ell,Y^{k}_m]_{x_0}|
\]
This clearly implies the statement.
\end{proof}

\subsection{Dynamical bounds}
By Lemma \ref{lem-orthlie} it is sufficient to find a subsequence \( k_m\to \infty \) such that
\[
|[Y^{k_m}_{i(k_m)},Y^{k_m}_{j(k_m)}]_{x_0}| \to 0
\]
as \( m\to \infty \),
as this would imply \eqref{zero} and thus our result.
This is the key step in the argument. Notice first that by
  \( (\star) \), there exists a subsequence $k_m\to\infty$ such that
\begin{equation}
\label{eq:subseq}
\frac{s^{k_m}_d(x_0) s^{k_m}_{d-1}(x_0)}{m(D\phi^{k_m}|_{F_{x_0}})} \to 0
\end{equation}
as \( m\to \infty \).
Our result then follows immediately from the following estimate which, together with Lemma \ref{lem-orthlie} and \eqref{eq:subseq} implies that \( |\Pi [Z, W]_{x_0}|=0 \). This is of course where we use in a crucial way the specific choice of local frames.

\begin{lem}\label{lem-upperbound}
There is a constant $K>0$ s.t up to passing to a subsequence of $k_m$ one has

$$| [Y^{k_{m}}_{i(k_{m)}},Y^{k_{m}}_{j(k_{m)}}]_{x_0}| \leq K \frac{s^{k_{m}}_{d}(x_0) s^{k_{m}}_{d-1}(x_0)}
{m(D\phi^{k_{m}}|_{F_{x_0}})}$$
\end{lem}

\begin{proof}
We divide the proof into two parts. First we explain how to choose the required subsequence \( k_m \to \infty \). Then we show that for such a subsequence we have the upper bound given in the statement.

Let \( k_m\to 0 \) be the subsequence such that \eqref{eq:subseq} holds and
 consider the sequence of images \( \phi^{k_m}(x_0) \) of the point \( x_0 \). By compactness of \( M \) this sequence has a converging subsequence and so, up to taking a further subsequence (which we still denote by $k_m$)  if necessary, we can assume that there exists a point \( y\in M \) such that $\lim_{m\rightarrow \infty}\phi^{k_m}(x)=y$. Fix $m_0$ large enough
s.t for all $m>m_0$, $\phi^{k_m}(x_0)$ lies in a coordinate chart around $y$
and let  $A=
\{\phi^{k_m}(x_0)\}_{m>m_0} \cup y$. Notice that A  is a compact set.
Now for each \( k_m \), let \( (i(k_m), j(k_m)) \) denote the "maximizing" pairs of indices defined above,   and let
\[\hat Y_{k_m} :=
\frac{D\phi^{k_m}_{x_0}[Y^{k_m}_{i(k_m)},Y^{k_m}_{j(k_m)}]_{x_0}}{|D\phi^{k_{m}}[Y^{k_m}_{i(k_m)},Y^{k_m}_{j(k_m)}]_{x_0}|}.
\]
Then the sequence of vectors
$\hat Y_{k _m}$  lies in the compact space $A \times S$, where $S$ is the unit ball in $\mathbb{R}^m$ and therefore there exists a subsequence of $k_m$ (which we still denote as $k_m$) and a
vector \( \hat Y\in T_yM \) such that
\[
\hat Y_{k_{m}} \to \hat Y
\]
as \( m \to \infty \).
Now recall  that by our choice of local frames we have
\[
[Y^{k_{m}}_{i(k_{m})},Y^{k_{m}}_{j(k_{m})}]_{x_0}\in F_{x_0}
\]
and since F is a $D\phi$ invariant and closed subset of $TM$ this implies also
 \[
 \hat Y_{k_{m}}\in F_{\phi^{k_{m}}(x_0)}  \quad \text{ and } \quad
 \hat Y\in F_y.
 \]

%Similiarly let
%$w^k_1(x),..,w^k_n(x)$ is an orthonormal choice of eigenvectors (up to degeneracies) at
%$E(x)$ which span the eigenspaces of $((D\phi^{-k}_x)|_{E})^{\dag}(D\phi^{-k}_x)|_{E}$ and which satisfy
%
%$$|(D\phi^{-k}_x)|_{E}w^k_i| = s^{-k}_i(x)$$
%
%where $s^{-k}_1(x) \geq s^{-k}_2(x) \geq ..$ are the singular values of the map $D\phi^{-k}_x|_{E}$.

Now let $E^{\perp}$ be the complementary subbundle of $TM$ orthonormal to $E$ and $\vartheta \subset T^*M$ be the subbundle defined by $g(E^{\perp},\cdot)$ which is the subbundle of $T^*M$ that defines $E$ by the orthogonality relation. For any $\eta, $ a section of $\vartheta$, one can write $\eta(\cdot) = g(V,\cdot)$ where $V$ is a section of $E^{\perp}$. Therefore since the bundle $F$ is uniformly bounded away from $E$ there exists a section $\eta$ of $\vartheta$ around $y$ s.t $\eta(\hat Y_{k_m}), \eta(\hat Y)>c$ for all $m$ large enough and for some constant $c>0$. In the following we assume that everything is evaluated at $x_0$, so we generally omit $x_0$ unless needed for clarity.

\begin{lem}
\label{cartanlem1} For every \(  m   \)   large enough we have
\[
  \eta(D\phi^{k_m}[Y^{k_m}_{i(k_m)},Y^{k_m}_{j(k_m)}]) \leq  |d\eta| s_{n-1}^{k_m}s_n^{k_m}.
 \]
\end{lem}
\begin{proof}
By the naturality of the Lie bracket we have
\begin{equation}\label{cartan1}
\eta(D\phi^{k_m}[Y^{k_m}_{i(k_m)},Y^{k_m}_{j(k_m)}])=
\eta([D\phi^{k_m}Y^{k_m}_{i(k_m)},D\phi^{k_m}Y^{k_m}_{j(k_m)}])
\end{equation}
We recall a formula in differential geometry (see \cite[Page 475]{Lee}), for any  two vector fields $Z,W$ and a
1-form $\eta$ we have
$$
\eta([Z,W])=Z(\eta(W)) - Z(\eta(W)) + d\eta(Z,W)
$$
Applying this formula to the right hand side of \eqref{cartan1} and using the fact that \(  \eta  \) is a bilinear form, and that \(  \eta(Y^{k_m}_{i(k_m)}) =
\eta(Y^{k_m}_{j(k_m)})=0  \) by construction, and the choice of the vectors  \(  v^{k_m}_{i(k_m)}, v^{k_m}_{j(k_m)}  \) (see Section \ref{locframes}) we get
\begin{align*}
|\eta([D\phi^{k_m}Y^{k_m}_{i(k_m)},D\phi^{k_m}Y^{k_m}_{j(k_m)}]_{x_0})|
&=
|d{\eta}(D\phi^{k_m}Y^{k_m}_{i(k_m)},D\phi^{k_m}Y^{k_m}_{j(k_m)})_{x_0}|
\\
&\leq
|d{\eta}| |D\phi^{k_m}_{x_0}v^{k_m}_{i(k_m)}|||D\phi^{k_m}_{x_0}v^{k_m}_{j(k_m)}|
 \\
 &\leq
|d{\eta}| s_{n-1}^{k_m}(x_0)s_n^{k_m}(x_0)
\end{align*}
Substituting into \eqref{cartan1} we get the result.
\end{proof}

\begin{lem}
\label{cartanlem2}
  For every \(  m   \)   large enough we have
\[
\eta(D\phi^{k_m}_{x_0}[Y^{k_m}_{i(k_m)},Y^{k_m}_{j(k_m)}]_{x_0}) \geq  |{\eta}(\hat Y_{k_{m}})|
m(D\phi^{k_m}_{x_0}|_F)| [Y^{k_m}_{i(k_m)},Y^{k_m}_{j(k_m)}]_{x_0}|
\]
\end{lem}
\begin{proof}
Notice first that by the definition of \(  \hat Y_{k_{m}}  \) we have
\[
|{\eta}(D\phi^{k_m}[Y^{k_m}_{i(k_m)},Y^{k_m}_{j(k_m)}])| =|D\phi^{k_m}[Y^{k_m}_{i(k_m)},Y^{k_m}_{j(k_m)}]| |{\eta}(\hat Y_{k_{m}})|
\]
Then, using the fact that
 $[Y^{k_m}_{i(k_m)},Y^{k_m}_{j(k_m)}]_{x_0} \in F(x_0)$  we get
$$|D\phi^{k_m}[Y^{k_m}_{i(k_m)},Y^{k_m}_{j(k_m)}]| \geq m(D\phi^{k_m}|_F)|  [Y^{k_m}_{i(k_m)},Y^{k_m}_{j(k_m)}]|
$$
Substituting this bound into the previous inequality we get the result.
\end{proof}

Returning to the proof of Lemma \ref{lem-upperbound}, combining  Lemmas \ref{cartanlem1} and  \ref{cartanlem2} we get
 \[
 |d\eta| s_{n-1}^{k_m}s_n^{k_m} \geq  |{\eta}(\hat Y_{k_{m}})|
m(D\phi^{k_m}|_F)| [Y^{k_m}_{i(k_m)},Y^{k_m}_{j(k_m)}]|.
\]
Using the fact that \(  |d\eta|  \) is bounded and  $|{\eta}(\hat Y_{k_{m}})|>c>0$ for all  \(  m  \) large enough there exists a constant \(  K>0  \) such that
$$|[Y^{k_m}_{i(k_m)},Y^{k_m}_{j(k_m)}]| < K \frac{s^{k_m}_{n}s^{k_m}_{n-1}}{m(D\phi^{k_m}|_{F})}.
$$
This completes the proof.
\end{proof}

\section{Lyapunov Regular Case}
\label{secvolpres}

In this Section we prove Theorem \ref{thm:volpres} by showing that the assumptions of the Theorem are equivalent to the assumptions of \cite[Theorem 1.2]{AH}, which has the same conclusions. To state this equivalence recall that by definition of regular point,  every \( x \in \mathcal A \) admits   
 an Oseledets splitting 
 \(
 E=E_1 \oplus E_2 \oplus ... \oplus E_n
 \)
for some \(  n\leq m  \), with associated Lyapunov exponents
\(
\lambda_1 <\lambda_2 < ... < \lambda_n.
\)
Let \( d_0=0 \),  \( d_n=m \) and, for each \( i=1,.., n-1 \), let
 \(
 d_i=\sum_{j=1}^{i}dim(E_j).
 \)

\begin{prop}\label{prop-lyasing}
For every \( i=1,..., n \) and every  $d_{i-1} < \ell \leq d_i$
we have
\[
\lim_{k\rightarrow \pm \infty}\frac{1}{k}\ln s^k_\ell = \lambda_{\ell}
\]
\end{prop}
Proposition \ref{prop-lyasing} says  that the rates of growth of the singular values are exactly the Lyapunov exponents. This implies that conditions 
 \((\star\star)\) is equivalent to 
\begin{equation}\label{loglim}
\lim_{k\rightarrow \infty} \frac{1}{k}\ln \frac{s^k_i(x)s^k_j(x)}{r^k_m(x)}  \neq 0.
\end{equation}
Letting 
\[
\mu_i(x)=\lim_{k\rightarrow \infty}\frac{1}{k} \ln{s^k_i(x)} \quad \mu_j(x)=\lim_{k\rightarrow \infty}\frac{1}{k} \ln s^k_j(x) \quad \lambda_m(x) = \lim_{k\rightarrow \infty}\frac{1}{k} \ln{r^k_m(x)}
\]
 equation \eqref{loglim} holds true if and only if $\mu_i(x) + \mu_j(x) \neq \lambda_m(x)$ for all indices $i,j,m$ which is exactly the condition given in \cite[Theorem 1.2]{AH}.

The proof of Proposition \ref{prop-lyasing} is based on the following two statements. The first is the so-called
Courant-Fischer MinMax Theorem. 

\begin{thm}[\cite{CH}]
Let $\phi: X \rightarrow Y$ be a linear operator between $n$ dimensional spaces. Let $s_1 \leq s_2 \leq ... \leq s_n$ be its singular values. Then
\begin{equation}\label{courantfischer}
s_\ell = \sup_{dim(V)=n-\ell+1}m(\phi|_V) = \inf_{dim(W)=\ell}|\phi|_V|
\end{equation}
where \( sup/inf \) above are taken over all subspaces $V,W$ of  the given dimensions.
\end{thm}

We refer the reader to \cite{CH} for the proof. 
The second statement we need is 

\begin{lem}\label{lem-singular}
Let $E_{\ell m}(x) = E_{\ell}(x) \oplus ... \oplus E_m(x)$ where $\ell < m$ and the associated Lyapunov exponents are ordered as
$\lambda_{\ell}(x) < \lambda_{\ell+1}(x) < ... <\lambda_m(x)$. Then
$$
\lim_{k\rightarrow \infty} \frac{1}{k} log(||D\phi^k_x|_{E_{\ell m}}||) = \lambda_m(x)
\quad \text{ and } \quad
\lim_{k\rightarrow \infty} \frac{1}{k} log(m(D\phi^k_x|_{E_{\ell m}})) = \lambda_{\ell}(x)$$
\end{lem}

Lemma \ref{lem-singular}
is  part of the Multiplicative Ergodic Theorem of Oseledets, and is stated in the  notes \cite{Boc} and is also stated and proved, though with a somewhat different notation than that used here, in \cite[Theorem 3.3.10]{Arn98}. 
%For completeness we include a self-contained proof below. First we show how these two statements imply our result. 

\begin{proof}[Proof of Proposition \ref{prop-lyasing}]
For every \( 1\leq  i \leq d  \)
and any subspaces \(  V', W'\subset E_{i}  \)
 with \(  dim(W') = \ell-d_{i-1} \) and \( dim (V')=d_i-\ell+1 \)
we write
$
W = E_1 \oplus E_2 \oplus ... \oplus E_{i-1} \oplus W'
$ and
$
V = V' \oplus E_{i+1} \oplus ... \oplus E_{n}
$
Then, by \eqref{courantfischer} we have
$
|D\phi^k|_W| \geq s^k_\ell \geq m(D\phi^k|_V).
$
Denote $E_{1\ell} = E_1 \oplus ... \oplus E_{\ell}$ and $E_{\ell n} = E_{\ell} \oplus ... \oplus E_n$.  Since $E_{1\ell} \supset W$ and ${E_{\ell n}}\supset V$  one has that $||D\phi^k|_{E_{1\ell}}|| \geq ||D\phi^k|_W||$ and $m(D\phi^k|_V) \geq m(D\phi^k|_{E_{\ell n}})$. Therefore
$ ||D\phi^k|_{E_{1\ell}}|| \geq s^k_{\ell} \geq m(D\phi^k|_{E_{\ell n}})$
and the result follows by Lemma \ref{lem-singular}.
\end{proof}

\bigskip
Stefano Luzzatto \\
\textsc{Abdus Salam International Centre for Theoretical Physics (ICTP), Strada Costiera 11, Trieste, Italy}\\
 \textit{Email address:} \texttt{luzzatto@ictp.it}

\medskip
Sina Tureli \\
\textsc{Abdus Salam International Centre for Theoretical Physics (ICTP), Strada Costiera 11, Trieste, Italy} and \textsc{International School for Advanced Studies (SISSA), Via Bonomea 265, Trieste}\\ 
 \textit{Email address:} \texttt{sinatureli@gmail.com}

\medskip
Khadim M. War\\
\textsc{Abdus Salam International Centre for Theoretical Physics (ICTP), Strada Costiera 11, Trieste, Italy} and \textsc{International School for Advanced Studies (SISSA), Via Bonomea 265, Trieste}\\ 
 \textit{Email address:} \texttt{kwar@ictp.it}

\end{document}